\documentclass[a4paper,12pt]{article}
\usepackage{amsthm}
\usepackage[utf8]{inputenc}
\usepackage{extpfeil}
\usepackage[
backend=biber,
style=numeric
]{biblatex}
\DeclareNameAlias{default}{last-first}
\usepackage{subcaption}
\usepackage{geometry}
\usepackage{lipsum}
\usepackage{tikz}
\usetikzlibrary{decorations.pathreplacing,shapes.misc}
\usepackage{pgfplots}
\pgfplotsset{compat=1.11}
\usepackage{mathtools}
\usepackage{changepage}
\usepackage{theoremref}
\newlength{\abstractwidth}
\renewcommand{\thanks}[1]{\footnote{#1}} 

\newcommand{\be}{\begin{equation}}
\newcommand{\bea}{\begin{eqnarray}}
\newcommand{\eea}{\end{eqnarray}} \newcommand{\ee}{\end{equation}}
 
 \def\ba{\begin{eqnarray}}
\def\ea{\end{eqnarray}}

\newcommand{\norm}[1]{\left\lVert#1\right\rVert}

\def\[{{\bf [}}
\def\]{{\bf ]}}

\begin{document}
\newtheorem{theorem}{Theorem} [section]
\newtheorem{proposition}[theorem]{Proposition} 
\newtheorem{lemma}[theorem]{Lemma} 
\newtheorem{corollary}[theorem]{Corollary} 
\newtheorem{definition}[theorem]{Definition} 
\newtheorem{conjecture}[theorem]{Conjecture} 
\newtheorem{example}[theorem]{Example} 
\newtheorem{claim}[theorem]{Claim} 
\newtheorem{remark}[theorem]{Remark} 
\newtheorem*{proposition*}{Proposition}

\begin{centering}
 
\textup{\LARGE\bf The Limit of the Inverse Mean \\
\vspace{0.3cm} Curvature Flow on a Torus}

\vspace{10 mm}

\textnormal{\large Brian Harvie} \\

\vspace{.5 in}
\begin{abstract}
{\small For an $H>0$ rotationally symmetric embedded torus $N_{0} \subset \mathbb{R}^{3}$ evolved by Inverse Mean Curvature Flow, we show that the total curvature $|A|$ remains bounded up to the singular time $T_{\max}$. This in turn implies convergence of the $N_{t}$ to a $C^{1}$ rotationally symmetric embedded torus $N_{T_{\max}}$ as $t \rightarrow T_{\max}$ without rescaling, contrasting sharply with the behavior of other extrinsic flows. Later, we note a scale-invariant $L^{2}$ energy estimate on any flow solution in $\mathbb{R}^{3}$ that may be useful in ruling out curvature blowup near singularities more generally.
}
\end{abstract}
\end{centering}
\section{Introduction}
Consider a standard round torus $N_{0} \subset \mathbb{R}^{3}$ with inner radius chosen small enough so that its mean curvature $H$ is everywhere positive. We would like to examine the evolution of $N_{0}$ under Inverse Mean Curvature Flow. Given an $n$-dimensional closed, oriented smooth manifold $N$, a one-parameter family of immersions $F: N \times [0,T) \rightarrow \mathbb{R}^{n+1}$ moves by Inverse Mean Curvature Flow (IMCF) if

\begin{equation} \label{IMCF}
    \frac{\partial N_{t}}{\partial t} (p,t) = \frac{1}{H} \nu(p,t)
\end{equation}
for each $(p,t) \in N \times [0,T)$, where $N_{t}=F_{t}(N)$ and $\nu$ and $H>0$ are the outward-pointing unit normal and mean curvature of $N_{t}$, respectively. Although in general singularities need not develop under IMCF (see \cite{claus}, \cite{huisken2}, \cite{Urbas}, \cite{me2}), a singularity will indeed form for the thin torus under the classical flow by the argument in Section 0 of \cite{huisken}: as the torus expands, the inner ring moves toward the central axis of rotation. Since there is a lower bound on flow speed for IMCF, the mean curvature along this ring must eventually reach zero, thereby terminating the flow. 

A previously unaddressed question is whether or not the hole of the torus will be completely filled by the singular time. That is, whether or not this torus will ``pinch" at its axis of rotation at the singular time $T_{\max}$, or if the flow terminates due to mean curvature approaching zero before this can happen. The latter possibility is quite intriguing, since if pinching did not occur the second fundamental form $A$ of $N_{t}$ would remain uniformly bounded in $L^{\infty}$ norm up to $T_{\max}$. This would suggest that, after passing to a subsequence in time, the $N_{t}$'s converge to a limit surface $N_{T_{\max}}$ as $t \rightarrow T_{\max}$. For other well-known extrinsic flows such as Mean Curvature Flow, Gauss Curvature Flow, and Surface Tension Flow, no such limit surface should exist at the time $T_{\max}$ without first rescaling.

We show in this paper that the curvature remains bounded and hence a limit surface does indeed exist under IMCF in this context. In fact, our proof applies more generally to any $H>0$ embedded torus in $\mathbb{R}^{3}$ with rotational symmetry about its central axis.

\begin{theorem} \thlabel{main}
Let $N_{0}=F_{0}(\mathbb{T}^{2}) \subset \mathbb{R}^{3}$ be an $H>0$, rotationally symmetric embedded torus and $F: \mathbb{T}^{2} \times [0,T_{\max}) \rightarrow \mathbb{R}^{3}$ the corresponding maximal solution to \eqref{IMCF}. Then $T_{\max} < +\infty$ and $\lim_{t \rightarrow T_{\max}} \max_{N_{t}} |A| \leq L < + \infty$. In particular, there exists a subsequence of times $t_{k} \nearrow T_{\max}$ and corresponding diffeomorphisms $\alpha_{k}: \mathbb{T}^{2} \rightarrow \mathbb{T}^{2}$ so that the maps $\widetilde{F}_{t_{k}}=F_{t_{k}} \circ \alpha_{k}: \mathbb{T}^{2} \rightarrow \mathbb{R}^{3}$ converge in $C^{1}$ topology to an immersion $\widetilde{F}_{T_{\max}}$. Furthermore, $\widetilde{F}_{T_{\max}}(\mathbb{T}^{2}) \subset \mathbb{R}^{3}$ is also a rotationally symmetric embedded torus.
\end{theorem}
Our argument by contradiction utilizes only elementary properties of the flow and an application of the Gauss-Bonnet Theorem.

This result raises the question of how general this behavior is for singular solutions of \eqref{IMCF}. As a first step toward answering this question, we also prove an $L^{2}$ energy estimate on $|A|$ which applies to any solution $N_{t}$ of \eqref{IMCF} in $\mathbb{R}^{3}$. As this estimate is also scale invariant, this may provide a way to rule out blow-up in $|A|$ altogether via rescaling arguments.

The paper is organized as follows: in Section 2, we consider the generating curve $\rho_{0}$ of a rotationally symmetric embedded torus $N_{0}$ in the $(x_{1},x_{2})$-plane, and we demonstrate that the generating curve $\rho_{t}$ of $N_{t}$ must remain embedded. This also ensures the formation of a singularity for $N_{0}$ under \eqref{IMCF} within a prescribed time interval by Corollary 2 in \cite{me}. In Section 3, we obtain a sharp upper bound on the $L^{1}$ norm of the Gauss Curvature $K$ over a uniform neighborhood of the ring of the expanding torus closest to the axis of rotation via the Gauss-Bonnet Theorem. We prove in Section 4 that $N_{t}$ cannot reach this axis by the time $T_{\max}$: if it did, we could rescale this neighborhood about this point and obtain convergence to a catenoid, which would contradict the integral bound on $K$. In Section 5, we apply a compactness theorem from \cite{Langer1985ACT} to rule out degeneration in the induced metric of $N_{t}$ near the singular time, so that the un-scaled $N_{t}$'s approach the embedded $C^{1}$ limit surface as $t \rightarrow T_{\max}$.

We consider singularities of any solution $N_{t}$ to \eqref{IMCF} in $\mathbb{R}^{3}$ in Section 6. The Gauss-Bonnet Theorem gives a uniform lower bound on the integral of $K$ for any such solution, which in turn controls the entire $L^{2}$ norm of $A$ independently of time. Since $n=2$, this estimate is scale invariant, and hence may be useful in ruling out blow-up in $\max_{N_{t}} |A|$ via rescaling arguments.

\section*{Acknowledgements}
I would like to thank my thesis advisor Adam Jacob, whose weekly discussions with me made this project possible, and the University of California, Davis Department of Mathematics for their financial support throughout my graduate studies.

\section{Preserving Embeddedness}
Once again, we consider an $H>0$ torus $N_{0} \subset \mathbb{R}^{3}$ which is obtained by revolving a simple closed curve in the upper half of the $(x_{1},x_{2})$-plane about the $x_{1}$ axis. Let $\{ e_{1}, e_{2} \}$ be the standard basis in $\mathbb{R}^{2}$, and $\nu$ be the outward normal of $\rho_{0}$ in the $(x_{1}, x_{2})$-plane. We parametrize $\rho_{0}$ by arc length, taking $s=0$ to be any point which minimizes the height $u=\langle \rho_{0}, e_{2} \rangle$. The principal curvatures of $N_{0}$ are the same along the ring in $\mathbb{R}^{3}$ generated by a point $\rho(s)$ on this curve. One of these principal curvatures $p(s)$ corresponds to rotation about the $x_{1}$ axis. This curvature is equal to

\begin{equation}
    p(s)= \langle \nu, e_{2} \rangle u^{-1} (s),
\end{equation}
and the other principal curvature $k(s)$ equals the curvature of $\rho_{0}$ in the plane at $s$. Then we write

\begin{equation}
    H(s)= k(s) + p(s)
\end{equation}
for the mean curvature $H$ of $N_{0}$. The assumption that $H$ is everywhere positive on $N_{0}$ gives us a strong profile for the curve $\rho_{0}$. Namely, $\rho_{0}$ must be the union of two graphs over the $x_{1}$ coordinate which correspond to the ``top" and ``bottom" of the curve.

\begin{proposition} \thlabel{graphs}
Let $\rho_{0}$ be the generating curve for an embedded, $H>0$, rotationally symmetric torus $N_{0} \subset \mathbb{R}^{3}$, and call $a_{0}=\min_{x \in \rho_{0}} \langle \rho_{0}, e_{1} \rangle(x)$ and $b_{0}=\max_{x \in \rho_{0}} \langle \rho_{0}, e_{1} \rangle(x)$. Then $\rho_{0}$ is the disjoint union of two graphs for functions $w_{0}: (a_{0},b_{0}) \rightarrow \mathbb{R}$ and $v_{0}: [a_{0},b_{0}] \rightarrow \mathbb{R}$ with $w_{0}(x) < v_{0}(x)$ over $(a_{0},b_{0})$. Furthermore, $\text{\upshape{graph}}(w_{0})$ is convex, and the outward unit normal $\nu$ of $\rho_{0}$ satisfies $\langle \nu, e_{2} \rangle <0$ on $\text{\upshape{graph}}(w_{0})$ and $\langle \nu, e_{2} \rangle>0$ on $\text{\upshape{Int}}(\text{\upshape{graph}}(v_{0}))$.
\end{proposition}

\begin{proof}
Once again, we parametrize by arc length so that $\rho_{0}(0)=\rho_{0}(\ell)$. By taking $\rho_{0}(0)$ to be a point which minimizes the height $u$, we know $\langle e_{2}, \nu \rangle (0) <0$, and hence this is also true in some neighborhood of this point. Let $[0,s_{1})$ and $(s_{2},\ell]$ be the largest intervals containing $0$ and $\ell$ respectively over which $\langle e_{2}, \nu \rangle <0$. We want to show that $\langle e_{2}, \nu \rangle (s) >0$ for each $s \in (s_{1},s_{2})$.
At any point $s_{0}$ where $\langle e_{2}, \nu \rangle$ vanishes, we have in view of the product rule and the fact that $\partial_{s}(s_{0})= \pm e_{2}$

\begin{equation} \label{arc}
    \frac{d}{ds} \langle e_{2}, \nu \rangle (s_{0})= \langle \nabla_{s} e_{2}, \nu \rangle + \langle e_{2}, \nabla_{s} \nu \rangle = \langle e_{2}, \nabla_{s} \nu \rangle  = \pm \langle \partial_{s}, \nabla_{s} \nu \rangle (s_{0})= \pm k(s_{0}). 
\end{equation}

Since $\langle e_{2}, \nu \rangle (s_{1})=0$, and hence $p(s_{1})=0$, $k(s_{1})>0$ in view of the mean convexity assumption. We then find $\frac{d}{ds} \langle e_{2}, \nu \rangle (s_{1})= + k(s_{1}) >0$. Therefore, $\langle e_{2}, \nu \rangle (s) >0$ in some right-handed neighborhood of $s_{1}$. Let $[s_{1}, s')$ be the largest such neighborhood over which $\langle e_{2}, \nu \rangle (s) >0$. 
We claim in fact that $s'=s_{2}$. Suppose not. We know that $\langle e_{2}, \nu \rangle (s') =0$ and $\frac{d}{ds} \langle e_{2}, \nu \rangle (s') \leq 0$. We consider the cases  $\partial_{s}(s')= \pm e_{2}$ separately, see \ref{curve}.

\begin{figure} 
    \centering
    \begin{tikzpicture}[xscale=1.7,yscale=1.7]
    \begin{scope}[shift={(-5,0)}]
    \begin{axis}[title={\tiny{Case I}}, ylabel=\tiny{$x_{2}$}, xlabel=\tiny{$x_{1}$}, y=0.5cm, x=0.5cm,
          xmax=4.2, xmin=-4.2, ymax=4.1, ymin=0,
          axis lines=middle,
          restrict y to domain=-7:20,
           xtick = {0},
        ytick= {0},
        yticklabels= {0, $u_{\min}(0)$}
          enlargelimits]
        \draw plot[smooth, ultra thick, color=black] coordinates{ (0,1) (1,1.2) (2.5,2.5) (2,3.2) (1,3.5) (0,3) (-0.5,2.8) (-1,3) (-0.5,3.5) (-0.8, 3.8) (-1.5, 4)  (-2.5,3.3) (-2.2,2.3) (-1.5,1.6) (-0.7,1.2) (0,1)};
        \draw (2.5,2.5) node{\Large{$\cdot$}};
        \draw (-2.55,3.25) node{\Large{$\cdot$}};
        \draw (0,1) node[anchor=north]{\tiny{$s=0$}};
        \draw (2.5,2.5) node[anchor=west]{\tiny{$s=s_{1}$}};
        \draw (-2.5,3.3) node[anchor=east]{\tiny{$s=s_{2}$}};
        \draw[color=blue][->] (-1,3)--(-1,3.5);
        \draw (-0.75,3) node[anchor=east]{\textcolor{blue}{\tiny{$\partial_{s}(s')$}}};
        \draw (-2,0.5) node{\tiny{\textcolor{blue}{$H(s') \leq 0$}}};
        \end{axis}
        
    \begin{scope}[shift={(5,0)}]
     \begin{axis}[ title={\tiny{Case II}},ylabel=\tiny{$x_{2}$}, xlabel=\tiny{$x_{1}$}, y=0.5cm, x=0.5cm,
          xmax=4.2, xmin=-4.4, ymax=4.1, ymin=0,
          axis lines=middle,
          restrict y to domain=-7:20,
           xtick = {0},
        ytick= {0},
        yticklabels= {0, $u_{\min}(0)$}
          enlargelimits]
        \draw[fill=gray!30, opacity=0.75] plot[smooth, ultra thick] coordinates{ (0,1) (1,1.2) (2.5,2.5) (2,3.5) (1,4) (0,3.8) (0.3,3.2) (0.7,3) (1,2.8) (0.7,2.5) (0, 2.4) (-1, 2.7)  (-1.6,3.3) (-1.5,3.5) (-1.4,3.7) (-1.7,3.9) (-2.5,3.6) (-2.8,3.2) (-2.3,2) (-1.5,1.5) (0,1)};
        \draw[->][color=red] (1,2.8) -- (0.5, 2.8);
        \draw (2.5,2.5) node{\Large{$\cdot$}};
        \draw (-2.8,3.2) node{\Large{$\cdot$}};
        \draw (0,1) node[anchor=north]{\tiny{$s=0$}};
        \draw (2.5,2.5) node[anchor=west]{\tiny{$s=s_{1}$}};
        \draw[->][color=blue] (0,3.8)--(0,3.3);
        \draw (0.25,3.8) node[anchor=east]{\tiny{\textcolor{blue}{$\partial_{s}(s')$}}};
        \draw[color=red] (1,1.2)--(1,4);
        \draw (0.7,2.8) node[anchor=west]{\tiny{\textcolor{red}{$s=s''$}}};
        \draw (2,0.5) node{\tiny{\textcolor{red}{$H(s'')\leq 0$}}};
        \draw (0,1.7) node{\tiny{$E_{0}$}};
        \draw (-2.7,3.2) node[anchor=east]{\tiny{$s=s_{2}$}};
        \end{axis}
    \end{scope}
    \end{scope}
    \end{tikzpicture}
    \caption{If $\langle e_{2}, \nu \rangle =0$ at some point between $s_{1}$ and $s_{2}$, one can locate a point where $\langle e_{2}, \nu \rangle =0$ and $k \leq 0$.}
    \label{curve}
\end{figure}
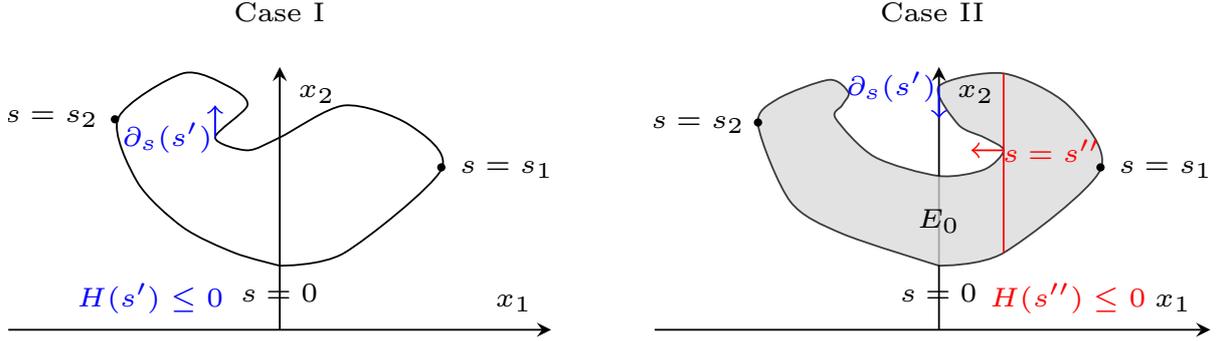

\textit{Case I: $\partial_{s}(s')=+e_{2}$}: \eqref{arc} implies $k(s') \leq 0$ in this case. But $p(s') =0$ here, so we would have altogether that $H(s') \leq 0$, a contradiction.

\textit{Case II: $\partial_{s}(s')=-e_{2}$}: If $\frac{d}{ds} \langle e_{2}, \nu \rangle (s') =0$, we would have $k(s')=0$ so that $H(s')=0$, which is a contradiction. Then assume $\frac{d}{ds} \langle e_{2}, \nu \rangle (s') < 0$. We also have

\begin{eqnarray*}
\frac{d}{ds} \rho_{1} (s') &=& 0, \\
\frac{d}{ds} \rho_{2} (s') &<& 0, \\
\frac{d^{2}}{ds^{2}} \rho_{1} (s') &>& 0.
\end{eqnarray*}
So there is some $\widetilde{s} \in (s', s_{2})$ with $\rho_{1}(\widetilde{s}) > \rho_{1}(s')$, $\rho_{2}(\widetilde{s}) < \rho_{2}(s')$. Now, letting $E_{0}$ be the open subset which $\rho_{0}$ encloses, define for each $s \in (s_{1},s_{2})$ the line $L_{s}$ by

\begin{equation*}
    L_{s} = \{ (x_{1},x_{2}) \in \overline{E}_{0} | x_{1}=\rho_{1}(s) \}.
\end{equation*}
Take the smallest $s$ value for which some interior point of $L_{s}$ intersects $\rho_{0}$, which must exist by the above observation and the fact that $k(s_{2})>0$. At the intersection point $\rho(s'')$ we have $\nu(s'')=-e_{1}$ and $k(s'') \leq 0$, which together imply $H(s'') \leq 0$. Once again by contradiction, we conclude $\langle e_{2}, \nu \rangle >0$ on $(s_{1},s_{2})$.

Now we know that the subsets $\{ \langle e_{2}, \nu \rangle < 0 \}$ and $\{ \langle e_{2}, \nu \rangle \geq 0 \}$ of $\rho_{0}$ are each comprised of a single connected component, and these must each be a graph over the $x_{1}$ coordinate. Note also that $k(s)>0$ wherever $\langle e_{2}, \nu \rangle (s) <0$ by the positive mean curvature assumption.

\end{proof}

We would like to consider the evolution of the torus $N_{0}$ generated by $\rho_{0}$ by \eqref{IMCF}. Since \eqref{IMCF} preserves rotational symmetry, we at least know that we can identify the flow surface $N_{t}$ with the curve $\rho_{t}$ in the $(x_{1},x_{2})$-plane which generates it. One cumbersome aspect of the analysis of IMCF is that, unlike with MCF, $N_{t}$ is not neccessarily embedded even if $N_{0}$ is. Our first task is then to rule out the possibility of self-intersections in the curve $\rho_{t}$. We may accomplish this using the profile for $\rho_{0}$ obtained in the above proposition. Often in the statement of our results, we will identify a solution of \eqref{IMCF} with the flow surfaces $N_{t}=F_{t}(N)$.

\begin{theorem}[Preserving Embeddedness] \thlabel{graphs}
Let $N_{0} \subset \mathbb{R}^{3}$ be a rotationally symmetric  $H>0$ embedded torus, and $\{N_{t}\}_{0 \leq t < T_{\max}}$ the corresponding maximal solution to \eqref{IMCF}. Then $N_{t}$ is embedded for each $t \in [0,T_{\max})$. In particular, each generating curve $\rho_{t}$ is the disjoint union of two graphs for $w_{t}$ and $v_{t}$ satisfying the conditions in \thref{graphs}.
\end{theorem}

\begin{proof}
$N_{t}$ always remains embedded for short time, so we demonstrate that if $t \nearrow t_{0}$ and $N_{t}$ is embedded for each $t < t_{0}$ then $N_{t_{0}}$ is necessarily embedded. Each generating curve $\rho_{t}$ for $N_{t}$ is the disjoint union of two graphs $w_{t}$ and $v_{t}$ with $w_{t} < v_{t}$ over the axis of rotation, so we will show this is also true for $\rho_{t_{0}}$. Let $E_{t}$ be the region enclosed by $\rho_{t}$. According to Theorem 4 in \cite{me}, we must have $E_{t_{1}} \subset E_{t_{2}}$ for $t_{1} < t_{2}$ in $[0,t_{0})$. This implies that $w_{t_{2}}(x) < w_{t_{1}}(x)$ and $v_{t_{2}}(x) > v_{t_{1}}(x)$ for any $x \in (a_{t_{1}},b_{t_{1}})$. Then for any compact set $K \subset (a_{t_{0}},b_{t_{0}})$ and $\widetilde{t} \in [0,t_{0})$ so that $K \subset (a_{\widetilde{t}},b_{\widetilde{t}})$, we have that $w_{t}|_{K}$ and $v_{t}|_{K}$ are each monotone and bounded over $t \in (\widetilde{t},t)$. The graphs of the limits $v_{t_{0}}|_{K}$ and $w_{t_{0}}|_{K}$ must also parametrize part of $N_{t_{0}}$ by uniqueness of limits and are therefore continuous. Then by Dini's Theorem $w_{t} \rightarrow w_{t_{0}}$ and $v_{t} \rightarrow v_{t_{0}}$ in $C^{0}_{\text{loc}}((a_{t_{0}},b_{t_{0}}))$ as $t \rightarrow t_{0}$. Note that $v_{t_{0}}(a_{t_{0}})= w_{t_{0}}(a_{t_{0}})$ since $w(a_{t})=v(a_{t})$, and likewise for $b_{t_{0}}$, so the union of  $\overline{\text{graph}(v_{t_{0}})}$ and $\text{graph}(w_{t_{0}})$ forms a closed curve. By uniqueness of limits, this union must equal $\rho_{t_{0}}$.

Now, we must have $\langle e_{2}, \nu \rangle \neq 0$ over the graphs of $w_{t_{0}}$ and $v_{t_{0}}$, since $\langle e_{2}, \nu \rangle \leq 0$ over $w_{t_{0}}$ (Resp. $\geq 0$ over $v_{t_{0}}$), and therefore if $\langle e_{2}, \nu \rangle (s_{0}) =0$ anywhere on $\text{graph}(w_{t_{0}})$ we would have by \eqref{arc} that




\begin{equation*}
    k(s_{0})= \frac{d}{ds} \langle e_{2}, \nu \rangle (s_{0}) = 0.
\end{equation*}
This would leave us with $H(s_{0})=0$ at this point, but we know $H(s) >0$ over $\rho_{t_{0}}$. The same argument yields that $\langle e_{2}, \nu \rangle >0$ over the graph of $v_{t_{0}}$. By the monotonicty of $w_{t}$ and $v_{t}$ noted above, we also have for any $x \in (a_{t_{0}}, b_{t_{0}})$ and $t$ sufficiently close to $t_{0}$ that

\begin{equation*}
    w_{t_{0}}(x) \leq w_{t}(x) < v_{t}(x) \leq v_{t_{0}}(x).
\end{equation*}
Hence these graphs do not intersect over $(a_{t_{0}}, b_{t_{0}})$. Since the graphs are disjoint, $\rho_{t_{0}}$ must be embedded.
\end{proof}

As mentioned in the above proof, embeddedness implies that $N_{t_{2}}$ must enclose $N_{t_{1}}$ whenever $t_{2}>t_{1}$ by Theorem 4 in \cite{me}. This, along with the convexity of the bottom graph of $\rho_{t}$, will be crucial for ruling out the possibility that $\lim_{t \rightarrow T} u_{\min}(t)=0$.

\section{An Energy Estimate on Gauss Curvature}

We now know by embeddedness and the topology of $N_{0}$ that $N_{t}$ must become singular at some time $T_{\max}$ which occurs within a prescribed time interval, see Corollary 2 in \cite{me}. In order to rule out the possibility of pinching, we first estimate the integral of the Gauss curvature around the inner ring of the torus closest to the axis of rotation. To this end, we first establish that $\langle e_{2}, \nu \rangle$ is bounded away from $0$ over some uniform neighborhood of the corresponding point on $\rho_{t}$. This neighborhood corresponds in $\mathbb{R}^{3}$ to the region of $N_{t}$ between two fixed parallel planes each perpendicular to the axis of rotation.

\begin{figure}[b!]
\centering
\begin{tikzpicture}[xscale=2,yscale=2]
\begin{axis}[
    y=0.5cm, x=0.5cm,
          xmax=4.4, xmin=-4.4, ymax=6, ymin=0,
          axis lines=middle,
          restrict y to domain=-7:20,
           xtick = {0},
        ytick= {0},
        yticklabels= {0, $u_{\min}(0)$}
          enlargelimits]
        \draw plot[smooth] coordinates{ (0,1) (1,1.5) (2.2,3) (2,4) (1,4.6) (0,5) (-0.5,4.6) (-1,4) (-2, 4.3)  (-2.5,3.8) (-2.6,3.3) (-2,2) (-0.7,1.2) (0,1)};
    
         \draw [color=blue, decorate,decoration={brace,amplitude=10pt}]
(0.75,1.2) -- (-0.75,1.2);
        \draw (0,0.3) node{\textcolor{blue}{\tiny{$S_{t}$}}};
        \draw[fill=gray!30,opacity=0.5] plot[smooth, ultra thick, color=black] coordinates{ (0,1.5) (0.5,1.7) (1,2.1) (1.3,2.6) (1.5,3.3) (1,3.7) (0.5,3.5) (0,3.7) (-0.5,3.6) (-1.2,3.3) (-1.7,2.8) (-1.1,2.1) (-0.5,1.7) (0,1.5)};
        \draw[dashed] (0.75,0) -- (0.75,6);
        \draw[dashed] (-0.75,0) -- (-0.75,6);
        \draw (0.75,5) node[anchor=west]{\tiny{$x_{1}=a$}};
        \draw (-0.75,5) node[anchor=east]{\tiny{$x_{1}=-a$}};
        \draw (0,2.5) node{\tiny{$E_{t_{0}}$}};
        \draw (2,3) node[anchor=west]{\tiny{$\rho_{t},t>t_{0}$}};
        \draw[color=red] (0,0.85)--(3,2.69);
        \end{axis}
\end{tikzpicture}
\caption{In order for $\rho_{t}$ to properly enclose $E_{t_{0}}$ for $t>t_{0}$, the tangent line must have a uniformly small slope on $\partial S_{t}$.}
\end{figure}
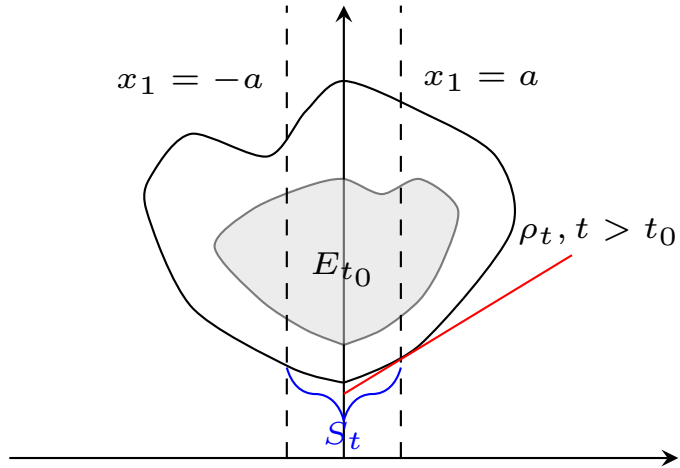

\begin{proposition} \thlabel{slope}
For a sequence of times $t_{n} \rightarrow T_{\max}$ and corresponding points $x_{n} \in \rho_{t_{n}}$ which minimize the height $u$, W.L.O.G. choose the $x_{2}$ axis so that $0=\lim_{n} \langle e_{1}, x_{n} \rangle$. Then there exists a constant $a>0$ such that the sets

\begin{equation*}
   S_{t}= \{ x \in \rho_{t} | \langle e_{2}, \nu \rangle (x) <0 \text{ \upshape{and} } |\langle e_{1}, \nu \rangle| (x) < a \}
\end{equation*}
are each graphs over $(-a,a)$ with $\langle e_{2}, \nu \rangle|_{\partial S_{t}} \leq c$ for some $c=c(N_{0})<0$. 
\end{proposition}

\begin{proof}
Take a $t_{0}$ sufficiently close to $T_{\max}$ so that $\min_{x \in \rho_{t_{0}}} \langle x, e_{1} \rangle <0$, $\max_{x \in \rho_{t_{0}}} \langle x, e_{1} \rangle >0$. Now define $a=\frac{1}{2} \min \{-\min_{x \in \rho_{{t}_{0}}} \langle x, e_{1} \rangle, \max_{x \in \rho_{t_{0}}} \langle x, e_{1} \rangle \}$. 

For $t> t_{0}$ sufficiently small, $\overline{S}_{t}$ may be parametrized by a convex graph $w_{t}: [-a,a] \rightarrow \mathbb{R}$ over the $x_{1}$ axis. We claim that $|w'_{t}(a)| \leq \frac{b}{a}$, where $b=\max_{x \in \rho_{t_{0}}} u(x)$. Suppose not: then the tangent line $L$ to $\rho_{t}$ at the point $(a,w_{t}(a))$ must pass through the region $E_{t_{0}}$ enclosed by the curve $\rho_{t_{0}}$, see Figure 2. By \thref{graphs}, $\rho_{t}$ is the union of disjoint graphs with the lower graph convex. This means that $\rho_{t}$ lies entirely on one side of $L$, so that $\rho_{t}$ also intersects $E_{t_{0}}$. According to Theorem 4 in \cite{me}, this cannot happen since the $\rho_{t}$ are embedded and hence must enclose $\overline{E}_{t_{0}}$ for $t>t_{0}$.

Thus, we have $|w'_{t}(a)| \leq \frac{b}{a}$ for $t > t_{0}$ sufficiently small. This also gives a time-independent bound on $w'_{t}$ over the entire domain by convexity, so $\overline{S_{t}}$ remains a graph over $[-a,a]$ with this uniform slope estimate on its boundary for each $t \in (t_{0},T_{\max})$. The result then follows from the relation $-\langle e_{2}, \nu \rangle = \frac{1}{(1+ |w'_{t}|^{2})^{\frac{1}{2}}}$, see, e.g. \cite{huisken1990} or \cite{ecker2}.
\end{proof}

Now, an application of Gauss-Bonnet gives a control on the $L^{1}$ norm of the Gauss curvature $K$ over the region of $N_{t}$ generated by $S_{t}$.

\begin{corollary}[Gauss Curvature Estimate] \thlabel{gauss}
Let $S_{t}$ be as above, and $S'_{t}$ the surface generated by revolving $S_{t}$ about the $x_{1}$ axis. Then for some constant $\epsilon=\epsilon(N_{0}) > 0$ we have

\begin{equation} \label{est}
    \int_{S'_{t}} |K| d \mu \leq 4\pi(1-\epsilon) 
\end{equation}
for each $t \in [0,T_{\max})$.
\end{corollary}
\begin{proof}
$S'_{t}$ is an embedded compact surface with boundary $\partial S'_{t}$ and an Euler Characteristic of $0$ in $\mathbb{R}^{3}$, so Gauss-Bonnet tells us that

\begin{equation*}
    \int_{S'_{t}} K d \mu = -\int_{\partial S'_{t}} k_{g} ds,
\end{equation*}
where $k_{g}$ is the geodesic curvature of $\partial S'_{t}$. Here, $\partial S'_{t}$ consists of two circles $C_{1}$ and $C_{2}$, call their radii $r_{1}$ and $r_{2}$ respectively. The geodesic curvatures over these circles are $k_{g_{1}}=-\frac{\langle e_{1}, \nu \rangle (x_{0})} {r_{1}}$ and $k_{g_{2}}=\frac{\langle e_{1}, \nu \rangle (y_{0})} {r_{2}}$, where $x_{0}$ and $y_{0}$ are the left and right endpoints of the curve $S_{t}$, respectively.

By \thref{slope} and the relation $|\langle e_{1}, \nu \rangle|^{2}= 1- |\langle e_{2}, \nu \rangle|^{2}$, we have $|\langle e_{1}, \nu \rangle| \leq 1-\epsilon$ for some uniform constant $
\epsilon >0$. Then

\begin{equation*}
    \int_{C} k_{g} ds = \int_{C_{1}} k_{g_{1}} ds + \int_{C_{2}} k_{g_{2}} ds \leq 4 \pi (1-\epsilon).
\end{equation*}

Noting that $K$ is strictly negative over $S'_{t}$, the result follows.
\end{proof}

Intuitively, this estimate is a promising sign that $|K|$, and hence $|A|$, remains uniformly bounded in $L^{\infty}$ norm near $T_{\max}$. We formally prove this in the following section, utilizing both the scale invariance of $\int_{S'_{t}} |K| d\mu$ as well as the rigidity for rotationally symmetric, complete minimal surfaces in $\mathbb{R}^{3}$.
\section{Rescaling the Singularity}
In this section, we derive a contradiction if $\lim_{t \rightarrow T_{\max}} u_{\min}(t)=0$ on the evolving curve $\rho_{t}$ to prove that $N_{t}$ converges to a smooth limit surface. 

Many of the tools developed to analyze the singularities of mean curvature flow do not translate to this setting. For example, the idea of a tangent flow is not applicable here since \eqref{IMCF} does not obey the standard parabolic scaling. Nevertheless, assuming $\lim_{t \rightarrow T} u_{\min}(t)=0$ in our setting, we may consider the re-scaled surfaces

\begin{equation}
    \widetilde{N}_{t}=\frac{1}{u_{\min}(t)}N_{t}.
\end{equation}
Crucially, mean curvature remains uniformly bounded above on $N_{t}$ by initial data under \eqref{IMCF} (see Section 6), which would imply $H \rightarrow 0$ uniformly over $\widetilde{N}_{t}$. Thus, any limit surface which the $\widetilde{N}_{t}$ (or some subset of each) converges to will necessarily be minimal. In the setting of a rotationally symmetric torus $N_{0}$, the candidate limit is rigid: we should expect convergence to a catenoid.

\begin{lemma} \thlabel{cat}
Let $S_{t} \subset \rho_{t}$ be as in \thref{slope}, and consider the rescalings $\widetilde{S}_{t}=\frac{1}{u_{\min}(t)} S_{t}$ with respect to the origin. Suppose $\lim_{t \rightarrow T} u_{\min}(t)=0$. Then the corresponding graphs $\widetilde{w}_{t}$ of $\widetilde{S}_{t}$ converge in $C^{2}_{loc}(\mathbb{R})$ to the function $\widetilde{w}(x)=\cosh(x)$ as $t \rightarrow T_{\max}$.
\end{lemma}

\begin{proof}
The parabolic maximum principle guarantees for any solution $\{ N_{t} \}_{0 \leq t< T_{\max}}$ to \eqref{IMCF} that $\max_{N_{t}} H \leq \max_{N_{0}} H$ (we give the relevant evolution equation in the proof of \thref{wilmore}). Therefore, assuming $\lim_{t \rightarrow T} u_{\min}(t)=0$, we also know for the surface $\widetilde{S}'_{t}$ generated by $\widetilde{S}_{t}$ that

\begin{equation}
    \lim_{t \rightarrow T_{\max}} \max_{\widetilde{S'}_{t}} H =0. 
\end{equation}

Consider $\overline{\widetilde{S}_{t}}$ to be a graph $\widetilde{w_{t}}: [-\widetilde{a}_{t},\widetilde{a}_{t}] \rightarrow \mathbb{R}$ over the $x_{1}$ coordinate, where $\widetilde{a}_{t}=\frac{1}{u_{\min}(t)}a_{t}$ for $a_{t}$ as in \eqref{slope}. Notice $\widetilde{a}_{t} \rightarrow \infty$ as $t \rightarrow T_{\max}$. We verify that a subsequence of the functions $\widetilde{w}_{t}$ converge in $C^{2}_{\text{loc}}(\mathbb{R})$ as $t \rightarrow T_{\max}$. Fix a compact subset $K$ of $\mathbb{R}$, and pick $t_{0}$ sufficiently close to $T_{\max}$ so that $K \subset [-\widetilde{a}_{t_{0}},\widetilde{a}_{t_{0}}]$. For any $t \in (t_{0},T_{\max})$, we know also from \thref{slope} that the function

\begin{equation*}
    v_{t}(x) = (1+ |\widetilde{w}'_{t}|^{2})^{\frac{1}{2}}
\end{equation*}
is uniformly bounded in $t$ over $[-\widetilde{a}_{t},\widetilde{a}_{t}]$. This guarantees convergence of a subsequence as $t \rightarrow T_{\max}$ at least in $C^{0}(K)$. Now, we can write the mean curvature of the surfaces generated by $\widetilde{w}_{t}$ as 

\begin{equation} \label{H_graph}
    \widetilde{H}(x)= \frac{\widetilde{w}''}{v^{3}} - \frac{1}{\widetilde{w}v}.
\end{equation}
Rearranging gives

\begin{equation} \label{second}
    \widetilde{w}''_{t}(x)= v_{t}^{3}(x) \widetilde{H}_{t}(x) + \frac{v^{2}_{t}(x)}{\widetilde{w}_{t}(x)}. 
\end{equation}
Noting that $\widetilde{w}_{t} \geq 1$, we can immediately see from this that $\widetilde{w}''$ is uniformly bounded in $t$, yielding precompactness in $C^{1}(K)$. In fact, we can observe equicontinuity of $\widetilde{w}''_{t}$ in $t$: $\widetilde{H}_{t}(x) \rightarrow 0$ uniformly over $K$ as $t \rightarrow T_{\max}$ by the uniform bound on $H$ and the assumption that $\lim_{t \rightarrow T_{\max}} u_{\min}(t)=0$. We also have

\begin{eqnarray*}
    |(\frac{1}{\widetilde{w}_{t}(x)})'|&=&\frac{1}{\widetilde{w}_{t}(x)^{2}} |\widetilde{w}'_{t}(x)| \leq C(N_{0}) \\
     |v_{t}'(x)|&=& |\frac{\widetilde{w}'_{t}\widetilde{w}''_{t}}{((1+ |\widetilde{w}'_{t}|^{2})^{\frac{1}{2}})}| \leq C(N_{0}).
\end{eqnarray*}
in view of the bound from below on $\widetilde{w}_{t}$ and the bounds from above on the first two derivatives. Therefore, $v_{t}$ and $\frac{1}{
\widetilde{w}_{t}}$ are each equicontinuous and bounded over $K$. Since \eqref{second} gives $\widetilde{w}_{t}''$ in terms of sums and products of $\widetilde{H}_{t}$, $v_{t}$, and $\frac{1}{\widetilde{w}_{t}}$, it is also bounded and equicontinuous in $t$, meaning the $\widetilde{w}_{t}$ are precompact in $C^{2,0}(K)$.

Pass to a subsequence in $t$ with a $C^{2}_{\text{loc}}(\mathbb{R})$ limit $\widetilde{w}$ if neccessary. Since \eqref{H_graph} uniformly approaches $0$ over $[-\widetilde{a}_{t},\widetilde{a}_{t}]$, $\widetilde{w}$ must satisfy

\begin{equation*}
    \widetilde{w}''(x)= \frac{(1 + |\widetilde{w}'|^{2})}{\widetilde{w}}(x).
\end{equation*}
The only solution to this differential equation over $\mathbb{R}$ is the catenary

\begin{equation*}
    \widetilde{w}(x)= \frac{1}{\gamma} \cosh(\gamma x)
\end{equation*}
for some $\gamma >0$. If not, there would exist a complete minimal surface of revolution in $\mathbb{R}^{3}$ which is not a catenoid. In this context, since convergence to $\widetilde{w}(x)$ is also pointwise, we must have

\begin{equation*}
    \widetilde{w}(0)=1,
\end{equation*}
implying $\gamma=1$ (the next theorem will not depend on the precise value of $\gamma$, but we scale so that $\gamma=1$ for the sake of simplicity).
\end{proof}

We can now derive a contradiction using the estimate \eqref{est}.

\begin{theorem}
Let $N_{0}$ be an $H>0$ rotationally symmetric embedded torus, and $\{N_{t}\}_{0 \leq t < T_{\max}}$ the corresponding solution to \eqref{IMCF}. Then $\lim_{t \rightarrow T_{\max}} u_{\min}(t) >0$, and hence $\lim_{t \rightarrow T_{\max}} \max_{N_{t}} |A| \leq L < +\infty$.
\end{theorem}

\begin{proof}
Suppose $\lim_{t \rightarrow T_{\max}} u_{\min}(t)=0$. According to the previous lemma, the functions $\widetilde{w}_{t}$ converge in $C^{2}_{\text{loc}}(\mathbb{R})$ to $\widetilde{w}(x)=\cosh(x)$. On the one hand, according to the Gauss curvature bound \eqref{est} which remains invariant under scaling, we should have

\begin{eqnarray} 
    \int_{\widetilde{S}'_{t}} |K| d\mu &=& 2\pi \int_{-\widetilde{a}_{t}}^{\widetilde{a}_{t}} (\frac{\widetilde{w}_{t}''}{(1+|\widetilde{w}_{t}'|^{2})^{\frac{3}{2}}})(\frac{1}{\widetilde{w}_{t}(1 + |\widetilde{w}_{t}'|^{2})^{\frac{1}{2}}}) \widetilde{w}_{t} (1+ |\widetilde{w}_{t}'|^{2})^{\frac{1}{2}} dx \label{small_int} \\ &=& 2\pi \int_{-\widetilde{a}_{t}}^{\widetilde{a}_{t}} \frac{\widetilde{w}_{t}''}{(1+ |\widetilde{w}_{t}'|^{2})^{\frac{3}{2}}} dx \leq 4\pi(1-\epsilon) \nonumber
\end{eqnarray}
for some $\epsilon=\epsilon(N_{0}) >0$. On the other hand, we may readily compute for $\widetilde{w}(x)=\cosh(x)$ that

\begin{equation*}
    2\pi \int_{\mathbb{R}} \frac{\widetilde{w}''}{(1+|\widetilde{w}'|^{2})^{\frac{3}{2}}} dx = 2\pi \int_{-\infty}^{\infty} \frac{\cosh(x)}{\cosh^{3}(x)} dx = 2\pi [\tanh(x)]^{\infty}_{\infty} = 4 \pi.
\end{equation*}
Fix a large enough interval $[-x_{0},x_{0}]$ so that

\begin{equation*}
2\pi \int_{-x_{0}}^{x_{0}} \frac{\widetilde{w}''}{(1+|\widetilde{w}'|^{2})^{\frac{3}{2}}} dx > 4\pi(1-\epsilon),
\end{equation*}
for $\epsilon$ is as in \eqref{small_int}. Then we would have $\frac{\widetilde{w}_{t}''}{(1+|\widetilde{w}_{t}'|^{2})^{\frac{3}{2}} }\rightarrow \frac{\widetilde{w}''}{(1+|\widetilde{w}'|^{2})^{\frac{3}{2}} }$ in $C^{0}([-x_{0},x_{0}])$ as $t \rightarrow T_{\max}$ but not in $L^{1}([-x_{0},x_{0}])$. This is a contradiction, so we cannot have the limit of $u_{\min}(t)$ equal to $0$. Then we also know for the original surface $N_{t}$ that

\begin{eqnarray*}
    \lim_{t \rightarrow T_{\max}} \max_{N_{t}} p &=& \lim_{t \rightarrow T_{\max}} \frac{1}{u_{\min}(t)} < +\infty, \\
    \lim_{t \rightarrow T_{\max}} \max_{N_{t}} k &\leq& \lim_{t \rightarrow T_{\max}} \max_{N_{t}} H + \lim_{t \rightarrow T_{\max}} \max_{N_{t}} p < +\infty.
\end{eqnarray*}
\end{proof}

A uniform-in-time bound on total curvature leads us to expect the existence of a smooth limit surface at the singular time without rescaling. Establishing this is actually nontrivial specifically in the context of IMCF, as we will explain in the next section.

\section{Convergence at $T_{\max}$}
Typically for extrinsic geometric flows, a uniform control on $\max_{N_{t}} |A|$ up to a time $T$ would imply uniform controls on all higher derivatives of $A$ as well as on the induced metric $g_{t}$ of $N_{t}$. These together would imply a smooth, non-degenerate limit surface $N_{T}$ at the time $T$. In the case of the rotationally symmetric expanding torus, the question of convergence at time $T=T_{\max}$ is more delicate: according to \cite{smoczyk} and Corollary 2.3 in \cite{huisken2}, singularities of IMCF are always characterized by the mean curvature $H$ of $N_{t}$ degnerating to $0$ somewhere. Therefore, although the total curvature $|A|$ remains bounded near $T_{\max}$ for this family of solutions, the flow speed $\frac{1}{H}$ is neccessarily blowing up. Since $\frac{1}{H}$ appears in the reaction terms of evolution equations for various geometric quantities, one cannot immediately establish $C^{\infty}$ convergence at the time $T_{\max}$ via a maximum principle.

We can, however, at least obtain $C^{1}$ convergence of the embeddings $F_{t}$ composed with appropriate diffeomorphisms to a map $\widetilde{F}_{T_{\max}}$ via the bound on $|A|$. Furthermore, according to a result from \cite{Langer1985ACT}, the induced metric $g_{t}$ cannot degenerate over as $t \rightarrow T_{\max}$ for any singular solution $\{ N_{t} \}_{0 \leq t < T_{\max}}$ to \eqref{IMCF} for which $\lim_{t \rightarrow T_{\max}} \max_{N_{t}} |A| < +\infty$.

\begin{proposition}
Let $F: N \times [0,T) \rightarrow \mathbb{R}^{3}$ be a solution to \eqref{IMCF} such that $T<+\infty$ and $\sup_{N \times [0,T)} |A| \leq L < +\infty$. Then for a subsequence of times $t_{k} \nearrow T$ and diffeomorphisms $\alpha_{k}: N \rightarrow N$, the immersions $\widetilde{F}_{t_{k}}= F_{t_{k}} \circ \alpha_{k}$ converge in $C^{1}$ topology to an immersion $\widetilde{F}_{T_{\max}}$ as $t \rightarrow T_{\max}$.
\end{proposition}

\begin{proof}
An elementary computation shows that the area under IMCF satisfies

\begin{equation*}
    \frac{d}{dt} |N_{t}| = \int_{N_{t}} \langle \frac{\partial N_{t}}{\partial t}, H \nu \rangle d \mu = \int_{N_{t}} d\mu = |N_{t}|,
\end{equation*}
meaning

\begin{equation} \label{area}
    |N_{t}|=e^{t} |N_{0}|.
\end{equation}
Therefore, the $L^{p}$ norm of $A$ for $p < +\infty$ is also uniformly bounded in time

\begin{equation*}
    \norm{A}_{p} = (\int_{N_{t}} |A|^{p} d\mu)^{\frac{1}{p}} \leq L^{\frac{1}{p}} e^{\frac{T}{p}} N_{0}.
\end{equation*}

According to the Compactness Theorem due to Langer from \cite{Langer1985ACT}, the maps $\widetilde{F}_{t_{k}}=F_{t_{k}} \circ \alpha_{k}$ converge in $C^{1}$ topology to an immersion $\widetilde{F}_{T_{\max}}$ as $t \rightarrow T_{\max}$. 

\end{proof}

In the context of the rotationally symmetric embedded torus, we can deduce that the limit surface should also be embedded. \\

\noindent \textit{Proof of \thref{main}}. Consider the generating curve $\rho_{t}$ of $N_{t}$. According to \thref{graphs}, for every $t \in [0,T_{\max})$ $\rho_{t}$ is the union of two disjoint graphs for two functions $w_{t}$ and $v_{t}$ which are each respectively monotone in time. Then repeating the same argument as in the proof of \thref{graphs}, $w_{t} \searrow w_{T_{\max}}$ and $v_{t} \nearrow v_{T_{\max}}$ in $C^{0}_{\text{loc}}(a_{T_{\max}}, b_{T_{\max}})$. Also by this argument, $w_{T_{\max}}(x) < v_{T_{\max}}(x)$ for $x \in (a_{T_{\max}}, b_{T_{\max}})$, and $\text{graph}(w_{T_{\max}}) \cup \text{graph}(v_{T_{\max}})$ is a closed curve. 

By the previous proposition, $\rho_{t}$ also converges to an immersed, closed $C^{1}$ curve $\rho_{T_{\max}}$ as $t \rightarrow T_{\max}$, so by uniqueness of limits $\rho_{T_{\max}}$ is embedded. Hence $\rho_{T_{\max}}$ generates an embedded torus in $\mathbb{R}^{3}$. 

\qed

\section{The General Case}
The natural question that this result raises is whether or not $|A|$ remains bounded and a limit surface will always exist at any singularity of \eqref{IMCF} without rescaling. Though we cannot currently provide a full answer to this question, we at least give some evidence that curvature will remain bounded over any singular solution of \eqref{IMCF} with $n=2$. We can once again obtain a stronger profile on more general flow behavior through an argument similar to the one in \thref{gauss}: for a solution $\{ N_{t}\}_{0 \leq t < T}$ in $\mathbb{R}^{3}$, Gauss-Bonnet will always provide a valuable energy estimate. First, we briefly note a simple $L^{2}$ estimate on mean curvature.


\begin{lemma} \thlabel{wilmore}
Let $\{ N_{t} \}_{0 \leq t <T}$ be a solution to \eqref{IMCF} in $\mathbb{R}^{3}$. Then $\int_{N_{t}} H^{2} d\mu \leq \sup_{N_{0}} H^{2} |N_{0}|$.
\end{lemma}
\begin{proof}
From \cite{huisken}, the evolution equation for $H$ under \eqref{IMCF} is

\begin{equation*}
    (\partial_{t} - \frac{1}{H^{2}} \Delta) H = -\frac{|A|^{2}}{H} - 2 \frac{|\nabla H|^{2}}{H^{3}}.
\end{equation*}
Therefore, for any $n$ we have for the function $f(x,t)= e^{\frac{t}{n}} H(x,t)$ that

\begin{equation*}
    (\partial_{t} - \frac{1}{H^{2}} \Delta) f = \frac{1}{n} f - \frac{|A|^{2}}{H^{2}} f \leq 0,
\end{equation*}
at any spacetime interior maximum of $f$ in $N \times [0,T)$, implying 

\begin{equation} \label{H}
    H(x,t) \leq e^{-\frac{t}{n}} \sup_{N_{0}} H
\end{equation}
by the parabolic maximum principle. Recalling equation \eqref{area} we also know $|N_{t}|=e^{t}|N_{0}|$. Then
combining this with \eqref{H} for the case $n=2$ yields

\begin{equation*}
\int_{N_{t}} H^{2} d\mu \leq (\sup_{N_{0}} H^{2}) |N_{0}|. 
\end{equation*}

\end{proof}

\begin{proposition}[$L^{2}$ Estimate on $|A|$] \thlabel{L2}
Let $\{ N_{t} \}_{0 \leq t < T}$ be a solution to \eqref{IMCF} in $\mathbb{R}^{3}$. Then we have the time-independent estimate

\begin{equation} \label{energy}
    \int_{N_{t}} |A|^{2} d\mu \leq 3 \sup_{N_{0}}H^{2} |N_{0}| - 2 \pi \chi(N),
\end{equation}
where $\chi(N)$ is the Euler Characteristic of $N$.
\end{proposition}
\begin{proof}
At any point $x \in N_{t}$ where the Gauss curvature $K$ of $N_{t}$ is positive, the two principal curvatures $\lambda_{1}(x)$ and $\lambda_{2}(x)$ of $N_{t}$ are also positive, since $\lambda_{1}(x) + \lambda_{2}(x)=H >0$ as long as the solution exists. In this case, we know that at these points

\begin{equation*}
\lambda_{1}(x) \leq H (x) \leq e^{\frac{-t}{2}} \sup_{N_{0}} H
\end{equation*}
in view of \eqref{H}, and likewise for $\lambda_{2}(x)$. This yields
\begin{equation}
    K \leq e^{-t} \sup_{N_{0}} H^{2}.
\end{equation}
Now write $K=K_{-} + K_{+}$, where $K_{-}=\min\{K,0\}$ and $K_{+}=\max\{K,0\}$. According to the above estimate, the area formula \eqref{area}, and Gauss-Bonnet, we have

\begin{equation*}
    \int_{N_{t}} K_{-} d \mu = 2\pi \chi(N_{0}) - \int_{N_{t}} K_{+} d\mu \geq 2\pi \chi(N) - \sup_{N_{0}} H^{2} |N_{0}|.
\end{equation*}
We remark that since the Gauss-Bonnet Theorem is intrinsic, see \cite{chern44}, this estimate applies to immersed solutions rather than exclusively to embedded ones. Finally, in $n=2$ we can write

\begin{equation*}
    \int_{N_{t}} |A|^{2} d\mu = \int_{N_{t}} H^{2} d\mu - 2 \int_{N_{t}} K d \mu.
\end{equation*}
The result follows in view of the previous lemma.
\end{proof}

\begin{remark}
Whether $\norm{A}_{L^{1}} < + \infty$ for any solution of MCF in $\mathbb{R}^{3}$ is still unknown, see \textup{\cite{du2020bounded}, \cite{head2013}, \cite{gianniotis2017diameter}}. Therefore, \eqref{energy} indicates better regularity in general near singularities for IMCF.
\end{remark}

This simple estimate suggests that, at the very least, there is likely no metric degeneration for any singular solution of IMCF in $\mathbb{R}^{3}$. Indeed, since $n=2$ this estimate corresponds to the borderline case for the Sobolev inequality, so an upgrading this to some $L^{p}$ estimate for some $p>2$ would allow one to conclude the existence of a limit immersion modulo diffeomorphisms by Langer's Compactness Theorem from \cite{Langer1985ACT}.

The estimate also hints that the procedure in section 4 may generalize to establish that $|A|$ may always remain bounded in $L^{\infty}$ norm. If $\lim_{t \rightarrow T_{\max}} \max_{x \in N_{t}} |A|(x)= + \infty$ then one would expect a sequence of times $t_{i} \rightarrow T_{\max}$ and corresponding scale factors $\lambda_{i} \rightarrow +\infty$ such that the surfaces $\widetilde{N}_{t_{i}}= \lambda_{i} N_{t_{i}}$ converge in $C^{2}$ to some non-compact limit $\widetilde{N}_{T_{\max}}$. Since \eqref{energy} is scale-invariant like with \eqref{est}, the estimate would also apply to this limit, and due to the upper bound on $\max_{N_{t}} H$ the limit $\widetilde{N}_{T_{\max}}$ must be minimal.

The estimates on mean and total curvature may allow one to rule out all candidate blowup limits of $\tilde{N}_{t}$ for scale factors tending do infinity, leading us to the following conjecture.

\begin{conjecture}
Let $\{N_{t}\}_{0 \leq t < T_{\max}}$ be a solution to \eqref{IMCF} in $\mathbb{R}^{3}$. Then $\max_{N \times [0,T_{\max})} |A| (x,t) < +\infty$.
\end{conjecture}
\appendix
\printbibliography

\begin{center}
\textnormal{ \large Department of Mathematics,
University of California, Davis \\
Davis, CA 95616\\
e-mail: bharvie@math.ucdavis.edu}\\
\end{center}
\end{document}